\documentclass[11pt]{amsart}

\usepackage[sc]{mathpazo}
\usepackage[margin=1in]{geometry}
\usepackage{amssymb}
\usepackage{hyperref}
\usepackage{amsmath,amscd}
\usepackage[enableskew]{youngtab}
\usepackage[all,cmtip]{xy}
\usepackage{qtree}

\newtheorem{theorem}{Theorem}[section]
\newtheorem{lemma}[theorem]{Lemma}
\newtheorem{proposition}[theorem]{Proposition}
\newtheorem{corollary}[theorem]{Corollary}

\theoremstyle{definition}

\newtheorem{example}[theorem]{Example}

\theoremstyle{remark}
\newtheorem{remark}[theorem]{Remark}

\def\RR{{\mathbb R}}
\def\CC{{\mathbb C}}
\def\FF{{\mathbb F}}

\DeclareMathOperator{\vspan}{span}

\def\T{\mathcal{T}}
\def\qand{\quad\text{and}\quad}

\newcommand\qbin[3]{\left[\begin{matrix} #1 \\ #2 \end{matrix} \right]_{#3}}

\begin{document}

\title{Nonassociativity of the Norton Algebras of some distance regular graphs} 
\author{Jia Huang}
\address{Department of Mathematics and Statistics, University of Nebraska, Kearney, NE 68849, USA}
\curraddr{}
\email{huangj2@unk.edu}

\begin{abstract}
A Norton algebra is an eigenspace of a distance regular graph endowed with a commutative nonassociative product called the Norton product, which is defined as the projection of the entrywise product onto this eigenspace. The Norton algebras are useful in finite group theory as they have interesting automorphism groups. We provide a precise quantitative measurement for the nonassociativity of the Norton product on the eigenspace of the second largest eigenvalue of the Johnson graphs, Grassmann graphs, Hamming graphs, and dual polar graphs, based on the formulas for this product established in previous work of Levstein, Maldonado and Penazzi. Our result shows that this product is as nonassociative as possible except for two cases, one being the trivial vanishing case while the other having connections with the integer sequence A000975 on OEIS and the so-called double minus operation studied recently by Huang, Mickey, and Xu.
\end{abstract}

\keywords{Nonassociativity, Norton algebra, Johnson graph, Grassmann graph, Hamming graph, dual polar graph}


\maketitle

\section{Introduction}
For any binary operation $*$ defined on a set $X$ with indeterminates $x_0, x_1, \ldots, x_n$ taking values from $X$, it is well known that the number of ways to insert parentheses into the expression $x_0* x_1*\cdots *x_n$ is the ubiquitous \emph{Catalan number} $C_n:= \frac{1}{n+1}\binom{2n}{n}$, which enumerates hundreds of other families of objects~\cite{EC2,Catalan}.
When $*$ is explicitly given, it is natural to ask, among all the $C_n$ ways to parenthesize the expression $x_0* x_1*\cdots *x_n$, what the exact number $C_{n,*}$ of distinct results is.
This problem had not received much attention until recently, when Hein and Huang~\cite{CatMod} proposed the study of $C_{n,*}$ as a quantitative measurement for the nonassociativity of a binary operation $*$, based on the observation that $1\le C_{*,n}\le C_n$ for all nonnegative integers $n$, where the first inequality is an equality if and only if the binary operation $*$ is associative.
When the other extreme occurs, i.e., $C_{*,n}=C_n$ for all $n\ge0$, we say that the binary operation $*$ is \emph{totally nonassociative}.
In general, the number $C_{*,n}$ measures the distance of $*$ from being associative or totally nonassociative.

Before work of Hein and Huang~\cite{CatMod}, Lord~\cite{Lord} introduced a measurement called the \emph{depth of nonassociativity} for a binary operation $*$, and examined it for some elementary binary operations.
It turns out that the depth of nonassociativity of $*$ can be written as $\inf \{ n+1: C_{*,n}<C_n\}$.
Thus it is substantially refined by the new measurement $C_{*,n}$.

Motivated by addition and subtraction, Hein and Huang~\cite{CatMod, VarCat} studied a large family of binary operations $*$ defined by using roots of unity, obtained explicit formulas for the number $C_{*,n}$ measuring the nonassociativity of $*$, and discovered connections to many Catalan objects with certain constraints.
Huang, Mickey, and Xu~\cite{DoubleMinus} determined the value of $C_{\ominus,n}$ for the \emph{double minus operation} $\ominus$ defined by $a\ominus b:= -a-b$, and discovered an coincidence between $C_{\ominus,n}$ and the interesting integer sequence A000975 in The On-Line Encyclopedia of Integer Sequences~\cite{OEIS}, which has many formulas and combinatorial interpretations (see also Stockmeyer~\cite{A000975}).

In this paper we study the nonassociativity of the so-called Norton algebras, whose construction relies on the notion of distance regular graph, an important topic in algebraic combinatorics~\cite{DistReg1, DistReg2, DistReg3}.
A \emph{distance regular graph} is a graph $\Gamma=(X,E)$ with vertex set $X$ and edge set $E$ such that for any two vertices $u$ and $v$, the number of vertices at distance $i$ from $u$ and at distance $j$ from $v$ depends only on $i$, $j$, and the distance between $u$ and $v$.
It is known that (the adjacency matrix of) $\Gamma$ has eigenvalues $\theta_0>\theta_1>\cdots>\theta_d$, where $d$ is the diameter of $\Gamma$, and the corresponding eigenspaces $V_0, V_1, \ldots, V_d$ form a direct sum decomposition $V=V_0\oplus V_1\oplus\cdots V_d$ for the vector space $V=\RR^X:=\{f: X\to \RR\}\cong\RR^{|X|}$.

For each $i=0,1,\ldots, d$, the \emph{Norton product} $\star$ on the eigenspace $V_i$ is defined by
\[ u\star v := \pi_i(u\cdot v) \quad\text{for all } u,v\in V_i \]
where $\pi_i$ is the orthogonal projection of $V$ onto $V_i$ and $\cdot$ is the entry-wise product given by the formula $(u\cdot v)(x):=u(x)v(x)$ for all $x\in X$.
The Norton product $\star$ is commutative but not necessarily associative.
With this product, each eigenspace $V_i$ becomes an algebra known as the \emph{Norton Algebra}.
The Norton algebras are useful in finite group theory as they have interesting automorphism groups~\cite{Norton1, Norton2} and are related to the construction of the monster simple group~\cite{Griess}.

We focus on the Norton product $\star$ on the eigenspace $V_1$ of the second largest eigenvalue of $\Gamma$, where $\Gamma$ a member of the following four important families of distance regular graphs: the Johnson graphs, Grassmann graphs, Hamming graphs, and dual polar graphs.

The \emph{Johnson graph} $J(n,k)$ has vertex set consisting of all $k$-subsets of $[n]:=\{1,2,\ldots,n\}$ and has edge set consisting of all unordered pairs $xy$ with $|x\cap y|=k-1$.
In particular, $J(n,1)$ is isomorphic to the complete graph $K_n$.
As a $q$-analogue of the Johnson graph $J(n,k)$, the \emph{Grassmann graph} $J_q(n,k)$ has vertex set consisting of all $k$-dimensional subspaces of a fixed $n$-dimensional vector space over the finite field $\FF_q$ and has edge set consisting of all unordered pairs $xy$ with $\dim(x\cap y)=k-1$.
We may assume $n\ge 2k$, without loss of generality, as taking the set complement (or orthogonal complement, resp.) gives a graph isomorphism $J(n,k) \cong J(n,n-k)$ (or $J_q(n,k)\cong J_q(n,n-k)$, resp.).
Both $J(n,k)$ and $J_q(n,k)$ are distance regular graphs with diameter $d=k$ (see, e.g., Brouwer, Cohen and Neumaier~\cite[\S~9.1, \S~9.3]{DistReg1}), and their eigenvalues and eigenspaces are well studied (see, e.g., Godsil and Meagher~\cite{GodsilMeagher}).

The \emph{Hamming graph} $H(d,e)$ has vertex set consisting of all words of length $d$ on the alphabet $\{1,2,\ldots,e\}$ (with $e\ge2$) and has edge set consisting of all unordered pairs $xy$ with $x$ and $y$ differing in exactly one position.
This is a distance regular graph with diameter $d$~\cite[\S~9.2]{DistReg1} and the special case $H(d,2)$ is the well-known \emph{hypercube graph}.

A \emph{dual polar graph} $\Gamma$ has vertex set $X$ consisting of all maximal isotropic subspaces of a certain finite dimensional vector space over a finite field $\FF_q$ with a nondegenerate form and has edge set $E$ consisting of unordered pairs $xy$ of vertices with $\dim(x\cap y)=k-1$.
It turns out that $\Gamma$ is a distance regular graph of diameter $d$, where $d:=\dim(x)$ does not depend on the choice of $x\in X$.
Since $d=1$ implies that $\Gamma$ is a complete graph, we may assume $d\ge 2$.
The dual polar graphs are commonly listed as $C_d(q)$, $B_d(q)$, $D_d(q)$, $D_{d+1}(q)$, $A_{2d}(r)$, $A_{2d-1}(r)$, where $r=\sqrt q$. 
These graphs already appeared as distance-transitive (hence distance regular) graphs in work of Hua~\cite{Hua} back in 1945.
See Brouwer, Cohen and Neumaier~\cite[\S~9.4]{DistReg1} for more details.

Levstein--Maldonado--Penazzi~\cite{DualPolarGraph} and Maldonado--Penazzi~\cite{NortonAlgebra} obtained formulas for the Norton product $\star$ on the eigenspace $V_1$ of $\Gamma$, where $\Gamma$ is the Johnson graph $J(n,k)$, the Grassmann graph $J_q(n,k)$, the hypercube graph $H(d,2)$ or a dual polar graph. 
Using these formulas together with an extended formula that we obtain from $H(d,2)$ to the Hamming graph $H(d,e)$, we determine the nonassociativity measurement $C_{\star,n}$.
Our main results are summarized below.

\begin{theorem}\label{thm1}
Let $\star$ be the Norton product on the eigenspace $V_1$ of the second largest eigenvalue of a distance regular graph $\Gamma$.
\begin{itemize}
\item
If $\Gamma$ is $J(2k,k)$ or $H(d,2)$ then $C_{\star,m}=1$ for all $m\ge0$, i.e., the operation $\star$ is associative.
\item
If $\Gamma$ is $J(3,1)$, $H(d,3)$, or $D_2(2)$ then 
$C_{\star,m}=C_{\ominus,m}$ for all $m\ge0$, which coincides with the OEIS sequence A000975~\cite{OEIS} except for $m=0$.
\item
If $\Gamma$ is $J(n,k)$ with $n>2k$ and $n\ge4$, $H(d,e)$ with $e\ge4$, or a dual graph polar graph of diameter $d\ge2$ other than $D_2(2)$, then $C_{\star,m}=C_m$ for all $m\ge0$, i.e., $\star$ is totally nonassociative.
\end{itemize}
\end{theorem}

The Norton products in Theorem~\ref{thm1} are either associative or totally nonassociative except for the second case.
This case is especially interesting as it provides a new interpretation for the sequence A000975 on OEIS~\cite{OEIS} with deep algebraic  and combinatorial background and is a natural higher-dimensional extension of the double minus operation coming from a somewhat surprising context.
In view of this, we believe that other Norton algebras are worth further investigation in the future.

This paper is structured as follows. 
In Section~\ref{sec:NonAssoc} we discuss the nonassociativity measurement for a binary operation, with an emphasis on the double minus operation.
In Section~\ref{sec:Norton} we focus on the formulas for the Norton algebras of distance regular graphs.
In Section~\ref{sec:main} we establish our main results for the Norton product on the eigenspace of the second largest eigenvalue of the Johnson graphs, Grassmann graphs, Hamming graphs and dual polar graphs.
We conclude the paper with some remarks and questions in Section~\ref{sec:remark}.

\section{Nonassociativity and binary trees}\label{sec:NonAssoc}

In this section we provide some results related to the nonassociativity measurement proposed by Hein and Huang~\cite{CatMod} and the double minus operation studied by Huang, Mickey, and Xu~\cite{DoubleMinus}.

Let $*$ be a binary operation defined on a set $X$.
Let $x_0, x_1, \ldots, x_n$ be $X$-valued indeterminates.
In general, the expression $x_0*x_1*\cdots*x_n$ is ambiguous, so we need to insert parentheses to specify the order in which the $*$'s are performed.
The parenthesizations of $x_0*x_1*\cdots*x_n$ are in bijection with binary trees with $n+1$ leaves, and thus enumerated by the ubiquitous \emph{Catalan number} $C_n:= \frac{1}{n+1}\binom{2n}{n}$.
Let $\T_n$ denote the set of all binary trees with $n+1$ leaves.
Given a tree $t\in\T_n$, let $(x_0*x_1*\cdots*x_n)_t$ denote the parenthesization of $x_0*x_1*\cdots*x_n$ corresponding to $t$.

For a specific binary operation $*$, it is possible that two parenthesizations of $x_0*x_1*\cdots*x_n$ are equal as functions from $X^{n+1}$ to $X$, and if so, the corresponding binary trees are said to be \emph{$*$-equivalent}.
Let $C_{*,n}$ denote the number of $*$-equivalence classes in the set $\T_n$.
It is clear that $C_{*,n}=1$ for all $n\ge0$ if and only if $*$ is associative, and in general, we have $1\le C_{*,n}\le C_n$.
Thus $C_{*,n}$ gives a quantitative measurement for the nonassociativity of the operation $*$.
We say that $*$ is \emph{totally nonassociative} if $C_{*,n} = C_n$ for all $n\ge0$.

Huang, Mickey and Xu~\cite{DoubleMinus} studied the \emph{double minus operation} on the complex field $\CC$ (or any other field in which $-1$ still has multiplicative order $2$) defined by $a\ominus b:=-a-b$ for all $a,b\in\CC$.
Parenthesizations for the double minus operation only depend on the leaf depth in binary trees.
Let $t\in\T_n$ and label its $n+1$ leaves $0, 1, \ldots, n$ from left to right (or more precisely, according to the preorder).
For each $i\in\{0,1,\ldots,n\}$, define the \emph{depth} $d_i(t)$ of leaf $i$ to be the length of the unique path from the root of $t$ to leaf $i$.
The \emph{depth sequence} of $t$ is $d(t):=(d_0(t), d_1(t), \ldots, d_n(t) )$.
One sees that
\[ (a_0\ominus a_1\ominus \cdots\ominus a_n)_t = (-1)^{d_0(t)} a_0 + (-1)^{d_1(t)} a_1 + \cdots (-1)^{d_n(t)} a_n. \]
Therefore two parenthesizations of $a_0\ominus a_1 \ominus \cdots \ominus a_n$ are equal if and only if the corresponding binary trees in $\T_n$ have (term-wise) congruent depth sequences modulo $2$.
This leads to the following result on the number $C_{\ominus, n}$.

\begin{theorem}\cite{DoubleMinus}\label{thm:DoubleMinus}
(i) Two binary trees $t,t'\in\T_n$ are $\ominus$-equivalent if and only if $d(t)\equiv d(t')\pmod 2$.

\noindent(ii) The sequence $(C_{\ominus,n})_{n=1}^\infty = (1,2,5,10,21,42,85, \ldots)$ coincides with OEIS sequence A000975~\cite{OEIS}.
\end{theorem}

The sequence A000975 in OEIS~\cite{OEIS} satisfies various recursive relations, such as $C_{\ominus,n+1}=2C_{\ominus,n}$ if $n$ is odd and $C_{\ominus,n+1}=2C_{\ominus,n+1}+1$ if $n$ is even.
It has the following formulas:  
\[ C_{\ominus,n}= \left\lfloor \frac{2^{n+1}}3 \right\rfloor = \frac{2^{n+2}-3-(-1)^n}{6} 
=\begin{cases} 
\displaystyle \frac{2^{n+1}-1}3, & \text{if $n$ is odd}; \\[15pt]
\displaystyle \frac{2^{n+1}-2}3, & \text{if $n$ is even}. 
\end{cases} \]
There are also a large number of combinatorial interpretations for the $n$th term of the sequence, including the number of steps required to solve the $n$-ring Chinese Rings puzzle, the distance between the all-zero string $0^n$ and all-one string $1^n$ in an $n$-bit binary Gray code, the positive integer with an alternate binary representation of length $n$, and so on. 
See Stockmeyer~\cite{A000975} and the references therein for details on this sequence.

While Theorem~\ref{thm:DoubleMinus} provides a different way of understanding the sequence A000975, we will give yet one more interpretation with more algebraic background by studying the Norton algebras of some distance regular graphs.
To this end, we need to make an observation on the depth sequence of a binary tree.
Define $\mathcal D_0:=\{(0)\}$ and for $n\ge0$ define
\[ \mathcal{D}_{n+1}:= \bigcup_{k=0}^n \{ (d_0+1,\ldots, d_k+1, d'_0+1, \ldots, d'_{n-k}+1): (d_0,\ldots,d_k)\in \mathcal{D}_{k},\ (d'_0,\ldots, d'_{n-k})\in \mathcal D_{n-k} \}. \]

\begin{proposition}\label{prop:depth}
Taking the depth sequence of a binary tree gives a bijection $d: \T_n \to \mathcal{D}_n$.
\end{proposition}

\begin{proof}
The result is trivial if $n=0$.
Assume it holds for $\T_n$, and we prove it for $\T_{n+1}$ below.

Any sequence in $\mathcal{D}_{n+1}$ can be written as $(d_0+1,\ldots, d_k+1, d'_0+1, \ldots, d'_{n-k}+1)$ for some sequences $(d_0,\ldots,d_k)\in \mathcal{D}_{k}$ and $(d'_0,\ldots, d'_{n-k})\in \mathcal D_{n-k}$, where $0\le k\le n$.
By the induction hypothesis, there exist trees $t\in\T_k$ and $t'\in \T_{n-k}$ such that $d(t) = (d_0,\ldots,d_k)$ and $d(t')=(d'_0,\ldots, d'_{n-k})$.
The unique binary tree with $t$ and $t'$ as the two subtrees under its root belongs to $\T_{n+1}$ and has depth sequence $(d_0+1,\ldots, d_k+1, d'_0+1, \ldots, d'_{n-k}+1)$.
Thus the map $d$ is onto.

Let $t\in\T_{n+1}$ with $d(t) = (d_0, d_1, \ldots, d_{n+1})$.
Let $i$ be the smallest integer such that $d_i$ is the largest among $d_0, d_1, \ldots, d_{n+1}$. 
Then $i$ is the leftmost leaf in $t$ with the largest depth among all of the leaves.
One sees that $i$ must be the left child of its parent, and its right sibling $i+1$ is another leaf with the same depth as $i$.
Deleting the two leaves $i$ and $i+1$ from $t$ gives a tree $t'\in\T_n$ with 
\[ d(t') = (d_0, \ldots, d_{i-1}, d_i-1, d_{i+2}, \ldots, d_{n+1}).\]

Now if $s\in T_{n+1}$ satisfies $d(s)=d(t)$, then using the same argument as above, we obtain $s'\in\T_n$ with $d(s')=d(t')$ by deleting the leaves $i$ and $i+1$.
The induction hypothesis implies $s'=t'$.
Since $s'$ and $t'$ are obtained from $s$ and $t$ in the same way, we conclude that $s=t$ and thus $d$ is one-to-one.
\end{proof}

It turns out that in some cases the Norton product can be viewed as a higher-dimensional extension of the double minus operation in the following sense.

\begin{lemma}\label{lem:OpTimes}
Given two binary operations $*$ and $\circ$ defined on two sets $R$ and $S$, respectively, define a new operation $\circledast$ on $R\times S$ by $(r,s)\circledast(r',s') := (r*r', s\circ s')$ for all $(r,s)\in R\times S$.
Then two binary trees are $\circledast$-equivalent if and only if they are both $*$-equivalent and $\circ$-equivalent.
\end{lemma}

\begin{proof}
Let $z_i = (r_i,s_i)$ be an arbitrary element of $R\times S$ for $i=0,1,\ldots,m$.
We have
\[ (z_0\circledast z_1\circledast \cdots \circledast z_m)_t = ( (r_0*r_1*\cdots*r_m)_t, (s_0\circ s_1\cdots\circ s_m)_t ) \]
for any binary tree $t\in\T_m$.
Thus for any $t, t'\in T_m$ we have
\[ (z_0\circledast z_1\circledast \cdots \circledast z_m)_t = (z_0\circledast z_1\circledast \cdots \circledast z_m)_{t'} \]
if and only if $(r_0*r_1*\cdots*r_m)_t=(r_0*r_1*\cdots*r_m)_{t'}$ and $(s_0\circ s_1\circ \cdots\circ s_m)_t = (s_0\circ s_1\circ \cdots\circ s_m)_{t'}$. 
This proves the desired result.
\end{proof}

\section{Distance regular graphs and Norton algebras}\label{sec:Norton}
In this section we summarize the results by Levstein, Maldonado and Penazzi~\cite{DualPolarGraph} and  Maldonado and Penazzi~\cite{NortonAlgebra} on the Norton algebras of certain distance regular graphs, and extend the result from the hypercube graphs to all Hamming graphs.
The reader is referred to Brouwer--Cohen--Neumaier~\cite{DistReg1} and van Dam--Koolen--Tanaka~\cite{DistReg2} for more background information on distance regular graphs.

\subsection{Distance regular graphs}
A graph $\Gamma=(X,E)$ with distance $d(-,-)$ is said to be \emph{distance regular} if for any integers $i,j,k\ge0$ and for any pair $(x,y)\in X\times X$ with $d(x,y)=k$, the number
\[ p_{ij}^k := \# \{z\in X: d(x,z)=i,\ d(y,z)=j\} \] 
is independent of the choice of the pair $(x,y)$.
The constants $p_{ij}^k$ are called the \emph{intersection numbers} of the distance regular graph $\Gamma$.

Let $\Gamma=(X,E)$ be a distance regular graph with diameter $d$.
Let $\mathcal M_X(\RR)$ denote the $\RR$-algebra of real matrices with rows and columns indexed by $X$.
For $0\le i\le d$, the \emph{$i$th adjacency matrix} $A_i$ of $\Gamma$ is the matrix in $M_X(\RR)$ whose $(x,y)$-entry is $1$ if $d(x,y)=i$ or $0$ otherwise.
In particular, $A=A_1$ is called the \emph{adjacency matrix} of the distance regular graph $\Gamma$;
this matrix is known to have eigenvalues $\theta_0>\theta_1>\cdots>\theta_d$ and the corresponding eigenspaces $V_0, V_1, \ldots, V_d$ form a direct sum decomposition $\RR^{X} = V_0\oplus V_1\oplus\cdots\oplus V_d$, where $\RR^X:=\{f: X\to\RR\} \cong \RR^{|X|}$.
We will simply call $\theta_0,\theta_1,\ldots,\theta_d$ and $V_0,V_1,\ldots,V_d$ the eigenvalues and eigenspaces of the graph $\Gamma$.

The \emph{adjacency algebra} $\mathcal A(\Gamma)$ of $\Gamma$ is the subalgebra of $\mathcal M_X(\RR)$ consisting of all polynomials in the adjacency matrix $A$ of $\Gamma$.
The primitive idempotents of this algebra are $E_0, E_1, \ldots, E_d$, where $E_i$ is the matrix of the orthogonal projection $\pi_i: \RR^X \to V_i$. 
The algebra $\mathcal A(\Gamma)$ has three important bases: $\{I, A, A^2, \ldots, A^d\}$, 
$\{A_0, A_1, \ldots, A_d\}$, and $\{E_0, E_1, \ldots, E_d\}$.



\subsection{The Johnson graphs}

Let $n$ and $k$ be two positive integers.
The \emph{Johnson graph} $J(n,k)=(X,E)$ has vertex set $X$ consisting of all $k$-subsets of the set $[n]:=\{1,2,\ldots,n\}$ and has edge set
\[ E=\{ xy : x, y\in X,\ |x\cap y|=k-1\}. \]
The graph Johnson $J(n,k)$ is a distance-regular graph since $d(x,y)=j$ if and only if $|x\cap y|=k-j$ for all $x,y\in X$.
For example, $J(n,1)$ is isomorphic to the complete graph $K_n$ and $J(n,2)$ is isomorphic to the line graph of $K_n$.

For any $k$-subsets $x$ and $y$ of $[n]$, one has $|x\cap y| = k-1$ if and only if $|x^c \cap y^c|=n-k-1$.
Thus $J(n,k)$ is isomorphic to $J(n,n-k)$, and we may assume $n\ge 2k$, without loss of generality.
The diameter of $J(n,k)$ is $d=k$, and for $i=0,1,\ldots,d=k$, the $i$th eigenvalue of the Johnson graph $J(n,k)$ is $\theta_i = (k-i)(n-k-i)-i$ 
whose multiplicity is
\[ \dim(V_i) = \binom{n}{i}-\binom{n}{i-1}. \]

To study the Norton algebras of $J(n,k)$, Maldonado and Penazzi~\cite{NortonAlgebra} constructed a lattice $L$ which consists of all subsets of $[n]$ with cardinality at most $k$ together with $\hat1:=[n]$.
The lattice $L$ is ordered by containment with minimum element $\hat0:=\emptyset$ and maximum element $\hat1$.
It has a rank function given by the cardinality of sets.
The formulas for the meet and join of $L$ are
\[ x\wedge y = x\cap y \qand x\vee y = 
\begin{cases}
x\cup y & \text{if } |x\cup y| \le k \\
\hat 1 & \text{ otherwise}.
\end{cases} \]

\subsection{The Grassmann graphs}
The \emph{Grassmann graph} $J_q(n,k)$ is a $q$-analogue of the Johnson graph $J(n,k)$.
Fix an $n$-dimensional vector space $\FF_q^n$ over the finite field $\FF_q$ with $q$ elements.
The vertex set $X$ of the graph $J_q(n,k)$ consists of all $k$-dimensional subspaces of $\FF_q^n$.
Two vertices are adjacent in $J_q(n,k)$ if and only if their intersection has dimension $k-1$.
More generally, we have $d(x,y)=j$ if and only if $\dim(x\cap y)=k-j$ for all $x,y\in X$. 

The \emph{orthogonal complement} of a subspace $z$ of $\FF_q^n$ is $z^\perp:=\{ w\in \FF_q^n: \langle z, w\rangle = 0 \}$
where we use the usual inner product $\langle z, w \rangle := z^t w$.
We have a graph isomorphism $J_q(n,k) \cong J_q(n,n-k)$ by taking the orthogonal complement since $\dim(x\cap y)=k-1$ if and only  $\dim(x^\perp \cap y^\perp) = n-k-1$.
Therefore we may assume $n\ge 2k$ for the Grassmann graph $J_q(n,k)$, without loss of generality.
We also assume $k\ge2$ as $J_q(n,1)$ is a complete graph which is already covered in the Johnson case. 

The Grassmann graph $J_q(n,k)$ is a distance regular graph with diameter $d=k$.
Many parameters of $J_q(n,k)$ are $q$-analogues of the Johnson graph $J(n,k)$.
Recall that an integer $m\ge0$ has its $q$-analogue defined by $[m]_q:=\frac{1-q^m}{1-q}=1+q+\cdots+q^{m-1}$.
The number of vertices in the Grassmann graph $J_q(n,k)$ is the $q$-binomial coefficient
\[ \qbin{n}{k}{q} := \frac{[n]_q[n-1]_q\cdots [n-k+1]_q}{[k]_q[k-1]_q\cdots [1]_q}. \]
For $i=0,1,\ldots,d=k$, the $i$th eigenvalue of the Grassmann graph $J_q(n,k)$ is 
\[ \theta_i=q^{i+1}[k-i]_q[n-k-i]_q-[i]_q \]
whose multiplicity is 
\[ \dim(V_i) = \qbin{n}{i}{q} - \qbin{n}{i-1}{q}.\]

Maldonado and Penazzi~\cite{NortonAlgebra} constructed a lattice $L$ which consists of all subspaces of $\FF_q^n$ with dimension at most $k$ together with $\hat1:=\FF_q^n$.
The lattice $L$ is ordered by containment with minimum element $\hat0:=0$ and maximum element $\hat1$.
The rank function of $L$ is given by the dimension of linear spaces.
The formulas for the meet and join of $L$ are
\[ x\wedge y = x\cap y \qand x\vee y = 
\begin{cases}
\vspan (x\cup y) & \text{if } \dim(\vspan (x\cup y))\le k \\
\hat 1 & \text{ otherwise}.
\end{cases} \]

\subsection{Hamming graphs}

The \emph{Hamming graph} $H(d,e)$ has vertex set $X$ consisting of all words of length $d$ on the alphabet $\{1,2,\ldots,e\}$ (where $e\ge2$) and has edge set $E$ consisting of all unordered pairs of vertices differing in precisely one position.
It is a distance regular graph of diameter $d$, with the distance between two vertices given by the number of positions in which they differ.
For $i=0,1,\ldots,d$, the $i$th eigenvalue of $H(d,e)$ is $\theta_i = (d-i)e-d$ and its multiplicity is $\binom{d}{i}(e-1)^i$~\cite[Theorem~9.2.1]{DistReg1}.
When $e=2$ the Hamming graph $H(d,2)$ is known as the \emph{hypercube graph}.

To study the Norton algebras of the Hamming graph $H(d,e)$, we construct a lattice $L$ which agrees with the lattice given by Maldonado and Penazzi~\cite{NortonAlgebra} in the special case of $e=2$.

For $i=0,1,\ldots,d$, let $L_i$ be the set of all words of length $d$ on the alphabet $\{0,1,\ldots,e\}$ with exactly $i$ nonzero entries.
For example, we have $L_0=\{ 0^d \}$ and $L_d=X$.

Let $u= u_1\cdots u_d$ and $v= v_1\cdots v_d$ be two words on the alphabet $\{0,1,\ldots,e\}$.
Define $u\le v$ if $u_i \ne 0 \Rightarrow v_i=u_i$ for all $i=1,\ldots,d$.
Then the disjoint union $L:=L_0\cup L_1\cup \cdots \cup L_d\cup\{\hat 1\}$ becomes a lattice with minimum element $\hat0:=0^d$ and maximum element $\hat1$ (not to be confused with the all-one word).
The rank of $u$ is the number of nonzero entries in $u$.
The $i$th entry of $u\wedge v$ is 
\[ (u \wedge v)_i = \begin{cases}
u_i & \text{if } u_i = v_i \\
0 & \text{if } u_i \ne v_i \\
\end{cases} 
\qquad\text{for } i=1,\ldots,d. \]
If $u_i\ne0$, $v_i\ne0$, and $u_i\ne v_i$ for some $i$, then $u\vee v = \hat 1$; 
otherwise $u\vee v$ has $i$th entry
\[ (u \vee v)_i = \begin{cases}
u_i & \text{if } u_i = v_i \\
u_i & \text{if } u_i \ne 0 = v_i \\
v_i & \text{if } u_i = 0 \ne v_i
\end{cases} 
\qquad\text{for } i=1,\ldots,d.  \]

\subsection{Dual polar graphs}
Let $V$ be a finite dimensional vector space over a finite field $\FF_q$ endowed with a nondegenerate form.
A subspace of $V$ is said to be \emph{isotropic} if the form vanishes on this subspace.
A \emph{dual polar graph} $\Gamma$ has vertex set $X$ consisting of all maximal isotropic subspaces of one of the following vector spaces, and has edge set consisting of all unordered pairs $xy$ of vertices with $\dim(x\cap y)=d-1$, where $d:=\dim(x)$ is independent of the choice of $x\in X$.
\hskip10pt
\begin{itemize}
\item
$C_d(q)$: $V=\FF_q^{2d}$ with a symplectic form; $e=1$.
\item
$B_d(q)$: $V=\FF_q^{2d+1}$ with a quadratic form; $e=1$.
\item
$D_d(q)$: $V=\FF_q^{2d}$ with a quadratic form of Witt index $d$; $e=0$.
\item
$D_{d+1}(q)$: $V=\FF_q^{2d+2}$ with a quadratic form of Witt index $d$; $e=2$.
\item
$A_{2d}(r)$: $V=\FF_q^{2d+1}$ with a Hermitian form, where $q=r^2$; $e=3/2$.
\item
$A_{2d-1}(r)$: $V=\FF_q^{2d}$ with a Hermitian form, where $q=r^2$; $e=1/2$.
\end{itemize}
The above dual polar graphs are all distance regular graphs with diameter $d$ and another important parameter $e$.
For $i=0,1,\ldots,d$, the $i$th eigenvalue of a dual polar graph is $\theta_i=q^e [d-i]_q - [i]_q$ 
and its multiplicity is~\cite[Theorem~9.4.3]{DistReg1}
\[ \dim(V_i) = q^i \qbin{d}{i}{q} \frac{1+q^{d+e-2i} }{1+q^{d+e-i} } \prod_{j=1}^i \frac{1+q^{d+e-j} }{1+q^{j-e} }. \]

Levstein, Maldonado and Penazzi~\cite{DualPolarGraph} constructed a lattice $L$ which consists of all isotropic subspaces of the underlying vector space $\FF_q^n$ together with the maximal element $\hat1:=\FF_q^n$.
The order, rank, meet, and join of the lattice $L$ are all similar to the Grassmann case, except that $u\vee v = \hat1$ if the span of $u\cup v$ is not isotropic.
The set $L_i$ of $i$-dimensional isotropic subspaces has cardinality~\cite[Lemma 9.4.1]{DistReg1}
\[ |L_i| = \qbin{d}{i}{q} \cdot \prod_{j=0}^{i-1} \left(q^{d+e-j-1}+1\right). \]

\subsection{Norton Algebras}
Let $\Gamma=(X,E)$ be a distance regular graph of diameter $d$, with eigenvalues $\theta_0>\theta_1>\cdots>\theta_d$ and corresponding eigenspaces $V_0,V_1,\ldots, V_d$.
For $i=0,1,\ldots,d$, using the orthogonal projection $\pi_i: \RR^{X} \to V_i$ we define the \emph{Norton product} on $V_i$ as
\[ u\star v := \pi_i(u\cdot v) \quad \text{ for all } u, v\in V_i \]
where $\cdot$ is the entrywise product, i.e., $(u\cdot v)(x) := u(x)v(x)$ for all $x\in X$.
With the Norton product $\star$, the eigenspace $V_i$ becomes an algebra known as the \emph{Norton algebra}, which is commutative but not associative in general.

Let $\Gamma=(X,E)$ be the Johnson graph $J(n,k)$, the Grassmann graph $J_q(n,k)$, the Hamming graph $H(d,e)$ or a dual polar graph throughout the rest of the paper.
Recall that there is a lattice $L$ associated with each of these graphs.
For any $v\in L$, define a map $\imath_v: X\to\RR$ by 
\[ \imath_v(x) :=
\begin{cases} 
1 & \text{if } v\le x \\ 0 & \text{otherwise}. 
\end{cases}\]
For $i=0,1,\ldots,d$, let $\Lambda_i$ denote the subspace of $\RR^X$ spanned by $\{\imath_v: v\in L_i\}$, where $L_i$ is the set of elements of rank $i$ in $L$.
In particular, $\Lambda_0$ is the span of the function $\mathbf{1}:X\to\RR$ which takes constant value $1$ on all vertices.
Also note that $\pi_i(\mathbf{1}) = \mathbf{0}$ is the zero function for $i=1,\ldots,d$.

Levstein--Maldonado--Penazzi~\cite{DualPolarGraph} and Maldonado--Penazzi~\cite{NortonAlgebra} showed the following result (whose proof in the hypercube case remains valid for all Hamming graphs).

\begin{theorem}\label{thm:filtration}
Let $\Gamma$ be $J(n,k)$, $J_q(n,k)$, $H(d,e)$ or a dual polar graph.

(i) There is a filtration $\Lambda_0\subseteq\Lambda_1\subseteq\cdots\subseteq\Lambda_d = \RR^X$.

(ii) The eigenspaces of $\Gamma$ are given by $V_0=\Lambda_0=\RR\mathbf{1}$ and $V_i=\Lambda_i\cap \Lambda_{i-1}^\perp$ for $i=1,2,\ldots,d$.

(iii) The set $\{ \check{v}: v\in L_1\}$ spans $V_1$, where $\check{v} := \pi_1(\imath_v)= \imath_v - \frac{a_1}{|X|} \mathbf{1}$ with $a_1:=\#\{ x\in X: x\ge v\}$ not depending on the choice of $v$.
\end{theorem}

For $\Gamma = J(n,k)$ with $n\ge2k$, Maldonado and Penazzi~\cite{NortonAlgebra} showed that if $u,v\in L_1$ then
\begin{equation}\label{eq:NortonProd1}
\displaystyle
\check{u} \star \check{v} = 
\begin{cases}
\displaystyle \left(1-\frac{2k}{n}\right) \check v & \text{if } u=v \\[9pt]
\displaystyle \frac{2k-n}{n(n-2)}(\check u + \check v) & \text{if } u\ne v.
\end{cases} 
\end{equation} 

For $\Gamma=J_q(n,k)$ with $n\ge2k\ge4$, Maldonado and Penazzi~\cite{NortonAlgebra} showed that if $u,v\in L_1$ then
\begin{equation}\label{eq:NortonProd2}
\check{u} \star \check{v} = 
\begin{cases}
\displaystyle
\left( 1- \frac{2[k]_q}{[n]_q} \right) \check{v} & \text{ if } u=v \\[9pt]
\displaystyle
-\frac{[k]_q}{[n]_q}(\check{u}+\check{v}) + \frac{[k-1]_q}{q[n-2]_q} \sum_{w\in L_1:\, w\le u\vee v} \check{w} 
& \text{ if } u\ne v.
\end{cases} 
\end{equation}
This is indeed a $q$-analogue of the previous formula~\eqref{eq:NortonProd1} since $-\frac{k}{n} +  \frac{k-1}{n-2} = \frac{2k-n}{n(n-2)}$.

For a dual polar graph $\Gamma$, Levstein, Maldonado and Penazzi~\cite{DualPolarGraph} showed that if $u,v\in L_1$ then
\begin{equation}\label{eq:DualPolar}
\check{u} \star \check{v} = 
\begin{cases}
\displaystyle
(q^{d+e-1}-1)\check{v}/(q^{d+e-1}+1) & \text{ if } u=v \\[5pt]
\displaystyle
-(\check{u}+\check{v})/(1+q^{d+e-1}) & \text{ if } u\vee v =\hat 1 \\[5pt]
\displaystyle
\frac{-(\check{u}+\check{v})}{1+q^{d+e-1}} + 
\sum_{w\in\Psi_2} \frac{\check w}{q^{d-1}(1+q^{e-1})} +
\sum_{w\in\Psi_3} \frac{\check w}{q^{d-1}(1+q^{e-1})(1+q^{d-3+e})} 
& \text{ otherwise}
\end{cases} 
\end{equation}
where $\Psi_j := \{w\in L_1: u\vee v\vee w \in L_j\}$ for $j=2,3$.

Finally, let $\Gamma=H(d,e)$.
Maldonado and Penazzi~\cite{NortonAlgebra} showed that the Norton product on $V_1$ is zero when $e=2$.
We generalize the result to all Hamming graphs.

\begin{theorem}\label{thm:Hamming}
Let $\Gamma=H(d,e)$.
For any $u,v\in L_1$, we have 
\begin{equation}\label{eq:NProdHamming}
\check u \star \check v = \begin{cases}
(e-2)\check v/e & \text{if } u = v \\
-(\check u + \check v)/e & \text{if } u\vee v = \hat 1 \\
\mathbf{0} & \text{if } u\vee v \in L_2.
\end{cases} 
\end{equation}
\end{theorem}

\begin{proof}
Let $u,v\in L_1$. 
By Theorem~\ref{thm:filtration}, we have
\[ \begin{aligned}
\check u \star \check v & = 
\pi_1 \left( (\imath_u - \mathbf{1}/e )\cdot (\imath_v-\mathbf{1}/e) \right) \\
&= \pi_1( \imath_u \cdot\imath_v) - \pi_1(\imath_u+\imath_v)/e +\pi_1(\mathbf{1})/e^2 \\
&= \pi_1( \imath_{u\vee v}) - (\check u+\check v)/e.
\end{aligned} \]

If $u=v$ then $\imath_{u\vee v} = \imath_v$ and thus $\check v \star \check v = \check v - 2\check v/e = (e-2)\check v/e$.

If $u\vee v =\hat 1$ then $\imath_{u\vee v} = \mathbf{0}$ and thus
\[ \begin{aligned}
\check u \star \check v 
&= \pi_1(\mathbf{0}) - (\check u+\check v)/e
=- (\check u+\check v)/e.
\end{aligned} \]

If $u\vee v \in L_2$ then $\check u\star \check v=\mathbf{0}$ since for any $w\in L_1$ we have
\[ \begin{aligned}
\langle \check u \star \check v, \pi_1(\imath_w) \rangle
&= \langle \imath_{u\vee v} - (\imath_u+\imath_v)/e,\ \imath_w - \mathbf{1}/e \rangle \\
&= \langle \imath_{u\vee v}, \imath_w \rangle - \langle \imath_u+\imath_v, \imath_w \rangle/e - \langle \imath_{u\vee v}, \mathbf{1} \rangle /e + \langle \imath_u+\imath_v, \mathbf1\rangle/e^2 \\
&= \langle \imath_{u\vee v}, \imath_w \rangle - \langle \imath_u+\imath_v, \imath_w \rangle/e - e^{d-3} + 2e^{d-3} = 0
\end{aligned}
\]
where the first equality holds by the orthogonality of $\pi_1$ and last one by the following argument.
\begin{itemize}
\item
If $w=u$ or $w=v$ then $\langle \imath_{u\vee v}, \imath_w \rangle=e^{d-2}$ and $\langle \imath_u+\imath_v, \imath_w \rangle = e^{d-1} +e^{d-2}$.
\item
If $w\vee u = \hat 1$ or $w\vee v = \hat 1$ then $\langle \imath_{u\vee v}, \imath_w \rangle=0$ and $\langle \imath_u+\imath_v, \imath_w \rangle = e^{d-2}$.
\item
If $u\vee v\vee w\in L_3$ then $\langle \imath_{u\vee v}, \imath_w \rangle = e^{d-3}$ and $\langle \imath_u+\imath_v, \imath_w \rangle = 2e^{d-2}$. \qedhere
\end{itemize}
\end{proof}

To better understand the Norton algebras of the Hamming graph $H(d,e)$, we provide a basis for each eigenspace.

\begin{proposition}\label{prop:Hamming}
For $i=0,1,\ldots,d$, the eigenspace $V_i$ has a basis $\{\pi_i(\imath_v): v\in L'_i\}$, where $L'_i$ is the set of all words in $L_i$ whose nonzero entries cannot be $e$.
In particular, $V_1$ has a basis $\{ \check v: v\in L'_1\}$ where $L'_1$ is the set of all words $v_1\cdots v_d$ with $1\le v_i<e$ for some $i$ and $v_j=0$ for all $j\ne i$. 
\end{proposition}

\begin{proof}
The result is trivial when $i=0$. Assume $i\ge1$ below.

We have a spanning set $\{ \pi_i(\imath_v): v\in L_i\}$ for $V_i$ by its definition.
Let $w\in L_i$ with $w_j=e$ for some $j$.
Changing the $j$th entry of $w$ to zero gives a word $u\in L_{i-1}$. 
For any $x\in X=L_d$, we have $x\ge u$ if and only if $x\ge v$ for some $v\gtrdot u$ with $v_j\ne0$.
Hence
\[ \sum_{v\gtrdot u:\ v_j \ne 0} \pi_i(\imath_v) = \pi_i(\imath_u) = 0. \]
This implies that we can write $\pi_i(\imath_w)$ in terms of $\pi_i(\imath_v)$ for all $v\gtrdot u$ with $1\le v_j<e$.
Repeating this process for all other entries of $w$ that equal $e$, we can write $\pi_i(\imath_w)$ in terms of $\{ \pi_i(\imath_v): v\in L'_i\}$.
Thus this set spans $V_i$ and it is indeed a basis since $\dim(V_i) = \binom{d}{i}(e-1)^i=|L'_i|$.
\end{proof}

\begin{corollary}\label{cor:Hamming}
Let $V_1(d,e)$ denote the Norton algebra $V_1$ of the Hamming graph $H(d,e)$.
Then $V_1(d,e)$ is isomorphic to the direct product of $d$ copies of $V_1(1,e)$: 
\[ V_1(d,e) \cong \underbrace{V_1(1,e) \times \cdots \times V_1(1,e) }_d .\]
\end{corollary}

\begin{proof}
By Proposition~\ref{prop:Hamming}, we have a basis $\{\bar{v}: v\in L'_1\}$ for $V_1(d,e)$, where $L'_1$ consists of all words of length $d$ on the alphabet $\{1,\ldots,e-1\}$.
By Theorem~\eqref{thm:Hamming}, the subalgebra spanned by $\{ \bar{v}: v\in L'_1,\ v_i\in\{1,\ldots,e-1\} \}$ is isomorphic to $V_1(1,e)$ for all $i=1,\ldots,d$ and $V_1(d,e)$ is isomorphic to the direct product of these subalgebras.
\end{proof}

\begin{remark}
One can also prove Corollary~\ref{cor:Hamming} by the method used in work of Levstein, Maldonado and D. Penazzi~\cite{Hamming} on the Terwilliger algebra of the Hamming scheme.
\end{remark}

\section{Main Results}\label{sec:main}

In this section we establish our main results on the nonassociativity of the Norton product $\star$ on the eigenspace $V_1$ of $\Gamma$, where $\Gamma$ is the Johnson graph $J(n,k)$, the Grassmann graph $J_q(n,k)$, the Hamming graph $H(d,e)$ or a dual polar graph.
Recall that $V_1$ has a spanning set $\{ \check{v} : v\in L_1 \}$ and a basis $\{ \check{v}: v\in L'_1\}$.
We also have formulas~\eqref{eq:NortonProd1}, \eqref{eq:NortonProd2}, \eqref{eq:DualPolar}, \eqref{eq:NProdHamming} for the Norton product $\star$ on $V_1$.

\subsection{The Johnson graphs}
In this subsection we study the Norton product $\star$ on the eigenspace $V_1$ of the Johnson graph $J(n,k)$.
When $n=2k$, the formula~\eqref{eq:NortonProd1} becomes $\check u\star \check v = 0$ for all $u,v\in L_1$ and thus $\star$ is associative.

We assume $n > 2k$ through the end of this subsection.
For each $v\in L_1$, let
\[ \bar v:= \frac{n}{n-2k} \check v.\]
Then $\{\bar v: v\in L_1\}$ is a spanning set for $V_1$.
Let $c := -1/(n-2)$. 
For any $u,v\in L_1$, we have 
\begin{equation}\label{eq:NortonProd1'}
\bar u \star \bar v = 
\begin{cases}
\bar v & \text{if } u=v \\
c(\bar u+\bar v) & \text{if } u\ne v
\end{cases} 
\end{equation}
by the formula~\eqref{eq:NortonProd1} for the Norton product $\star$ on the spanning set $\{ \check v: v\in L_1\}$ of $V_1$.

\begin{example}\label{example1}
For $n\ge3$, the Johnson graph $J(n,1) \cong K_n$ has vertices labeled by $1$-subsets of $[n]$.
Its adjacency matrix is $A = J-I$, whose eigenvalues are $\theta_0=n-1$ and $\theta_1=-1$.
We have $V_0= \RR\mathbf{1}$ and $V_1 = V_0^\perp = \left\{ f: X\to \RR: \sum_{x\in X} f(x) = 0 \right\}$.
For each $v=\{i\}\in L_1=X$ we have 
\[ \imath_v = e_i, \quad \check v:=\pi_1(e_i)=e_i-\mathbf{1}/n, \qand x_i := \bar v = n\check v/(n-2) \]
where $e_i: X\to\RR$ is defined by $e_i(\{i\}) = 1$ and $e_i(\{j\}) = 0$ for all $j\ne i$.
The set $\{x_1,\ldots,x_n\}$ spans $V_1$, and deleting any element from it gives a basis for $V_1$.
We have
\[ x_i \star  x_i = x_i \qand
x_i \star  x_j = c(x_i+x_j) \]
where $1\le i \ne j \le n$ and $c := -1/(n-2)$.
For distinct $i,j,k\in\{1,2,\ldots,n\}$, we have 
\[ x_i \star  (x_i \star  x_j) = c(x_i + x_i\star x_j), \]
\[ (x_i\star x_j) \star  (x_i\star x_j) = (2c^2+c)(x_i\star x_j).\]
\end{example}

The next lemma will play an important role in our study of the Norton product $\star$.

\begin{lemma}\label{lem:ij}
Let $t\in\T_m$. 
Let $u$ and $v$ be distinct elements of $L_1$.

\noindent(i)
If $z_0=z_1=\cdots=z_m=\bar v$ then $(z_0\star z_1\star\cdots\star z_m)_t=\bar v$.

\noindent(ii) 
If $z_r:=\bar u$ for some $r$ and $z_s=\bar v$ for all $s\ne r$, then 
\begin{equation}\label{eq:ij}
(z_0\star z_1\star\cdots\star z_m)_t = c^{d_r(t)} \bar u + \left( c+c^2+\cdots+c^{d_r(t)} \right) \bar v.
\end{equation}
\end{lemma}

\begin{proof}
(i) If $z_0=z_1=\cdots=z_m=\bar v$ then $(z_0\star z_1\star\cdots\star z_m)_t=\bar v$ since $\bar v\star \bar v=\bar v$ by the formula~\eqref{eq:NortonProd1'}. 

(ii) We use induction on $m$. The result is trivial when $m=1$.
Assume $m\ge2$ below.
Let $t_1\in\T_{m'}$ and $t_2\in\T_{m-m'-1}$ be the subtrees of $t$ rooted at the left and right children of the root of $t$, respectively.
Suppose the $r$th leaf of $t$ is contained in $t_1$, without loss of generality. 
Then $d_r(t)=d_r(t_1)+1$.
By the inductive hypothesis and part (i) of this lemma, we have
\begin{align*}
(z_0\star z_1\star\cdots\star z_m)_t 
& = (z_0\star\cdots\star z_{m'})_{t_1} \star (z_{m'+1}\star \cdots\star z_m)_{t_2} \\
& = \left(  c^{d_r(t_1)} \bar u + \left(c+c^2+\cdots+c^{d_r(t_1)} \right) \bar v \right) \star \bar v \\
& = c^{d_r(t_1)+1}(\bar u+\bar v) + \left(c+c^2+\cdots+c^{d_r(t_1)} \right) \bar v \\
& = c^{d_r(t)} \bar u + \left(c+c^2+\cdots+c^{d_r(t)} \right) \bar v.
\qedhere
\end{align*}
\end{proof}

It turns out the the case $n=3$ is different from the case $n\ge4$.
We first consider the former, which implies $k=1$ as we assume $n>2k$.
As discussed in Example~\ref{example1}, for the Johnson graph $J(3,1)\cong K_3$, the eigenspace $V_1$ of $\theta_1=-1$ has a basis $\{x,y\}$ which satisfies
\begin{equation}\label{eq:NortonProd3}
x\star x = x, \quad
y\star y = y, \qand
x\star y = -x-y.
\end{equation}
It follows that
\begin{equation}\label{eq:NortonProd3'}
x\star (x\star y) = y, \quad y\star (x\star y) = x, \qand (x\star y)\star (x\star y) = x\star y.\end{equation}

\begin{proposition}\label{prop:J31}
Let $\star$ be the Norton product on the eigenspace $V_1$ of the Johnson graph $J(3,1)\cong K_3$.
Then two binary trees in $\T_m$ are $\star$-equivalent if and only if their depth sequences are (term-wise) congruent modulo $2$.
Consequently, $C_{\star,m}=C_{\ominus,m}$ for all $m\ge0$, which agrees with the sequence A000975 on OEIS~\cite{OEIS} except for $m=0$.
\end{proposition}

\begin{proof}
Let $t$ and $t'$ be two binary trees in $\T_m$.
Let $z_0, \ldots, z_m$ be indeterminates taking values in $V_1$.
It suffices to show that $(z_0\star \cdots \star z_m)_{t} = (z_0\star \cdots \star z_m)_{t'}$ if and only if $d(t) \equiv d(t') \pmod 2$; the rest of the result will follow from Theorem~\ref{thm:DoubleMinus}.

First suppose that $(z_0\star \cdots \star z_m)_{t} = (z_0\star \cdots \star z_m)_{t'}$.
Let $r$ be an arbitrary element of $\{0,1,\ldots,m\}$.
By Lemma~\ref{lem:ij}, taking $z_r=x$ and $z_s=y$ for all $s\ne r$ gives 
\[ (-1)^{d_r(t)} x + \frac{1-(-1)^{d_r(t)+1}}2 y = (-1)^{d_r(t')} x + \frac{1-(-1)^{d_r(t')+1}}2 y. \]
This implies that $d_r(t) \equiv d_r(t') \pmod 2$.
Since $r$ is arbitrary, we have $d(t) \equiv d(t') \pmod 2$.

For the reverse direction, we compare $\star$ with the double minus operation $\ominus$, which can be defined on $V_1$  by $a\ominus b:=-a-b$ for all $a,b\in V_1$. 
Theorem~\ref{thm:DoubleMinus} still holds since $-1$ is a second root of unity in $\RR$.
Both $\ominus$ and $\star$ are commutative, and one can check that 
\begin{equation}\label{eq:mod3}
\begin{aligned}
x\ominus y &= -x-y = x\star y, \\
x\ominus x &= -2x \equiv x = x\star x \pmod 3, \\
y\ominus y &= -2y \equiv y = y\star y \pmod 3, \\
x\ominus(x\ominus y) &= -x-(-x-y) = y = x\star(x\star y), \\
y\ominus(x\ominus y) &= -y-(-x-y) = x = y\star(x\star y), \ \text{and} \\
(x\ominus y)\ominus(x\ominus y) &= 2x+2y \equiv -x-y=x\star y \pmod 3.
\end{aligned}
\end{equation}
Suppose that $d(t) \equiv d(t') \pmod 2$.
This implies that $(z_0\ominus \cdots \ominus z_m)_{t} =  (z_0\ominus \cdots \ominus z_m)_{t'}$.
To show $(z_0\star \cdots \star z_m)_{t} = (z_0\star \cdots \star z_m)_{t'}$, we may assume that $z_0, \ldots, z_m$ take values from the basis $\{x,y\}$ of $V_1$ by the linearity of the Norton product $\star$.
Using the formula~\eqref{eq:NortonProd3} for the Norton product $\star$ on $\{x,y\}$, we can expand $(z_0\star \cdots \star z_m)_{t}$ and $(z_0\star \cdots \star z_m)_{t'}$.
During this process, we will only encounter $x$, $y$, and $x\star y$, according to the formula~\eqref{eq:NortonProd3'}.
By reduction modulo $3$ and using the above relations~\eqref{eq:mod3}  between $\ominus$ and $\star$ we have $(z_0\star \cdots \star z_m)_{t} =  (z_0\star \cdots \star z_m)_{t'}$ from $(z_0\ominus \cdots \ominus z_m)_{t} =  (z_0\ominus \cdots \ominus z_m)_{t'}$.
\end{proof}

\begin{remark}
This proposition suggests that the Norton product $\star$ on $V_1$ for $J(3,1)\cong K_3$ can be viewed as a 2-dimensional generalization of the double minus operation $\ominus$.
The two operations are related by congruence modulo $3$, as shown in the above proof.
However, the two operations are not the same even if the ground field $\RR$ is replaced with a field of characteristic $3$.
For example, we have $(-x)\star y = -(x\star y) =x+y$ but $(-x)\ominus y = x-y$.
It would be nice to have an explicit formula for the result from expanding $(z_0\star \cdots \star z_m)_{t}$ for any tree $t\in\T_m$, where $z_0, \ldots, z_m$ take values in the basis $\{x,y\}$; such a formula may lead to a different proof of the above proposition.
\end{remark}

Now we study the case $n\ge4$, which is different from the previous case $n=3$, since in the formula~\eqref{eq:NortonProd1'} for the Norton product $\star$, the constant $c:=-1/(n-2)$ generates an infinite multiplicative group in the field $\RR$ when $n\ge4$.

\begin{proposition}\label{prop:J4}
Suppose $n\ge4$ and $n>2k$.
Let $\star$ be the Norton product on the eigenspace $V_1$ of the Johnson graph $J(n,k)$. 
Then two binary trees in $\T_m$ are $\star$-equivalent if and only if they are equal.
Consequently, $C_{\star,m}=C_m$ for all $m\ge0$, i.e., the operation $\star$ is totally nonassociative.
\end{proposition}

\begin{proof}
It suffices to show that any two distinct binary trees $t$ and $t'$ in $\T_m$ are not $\star$-equivalent.
By Proposition~\ref{prop:depth}, their depth sequences $d(t)$ and $d(t')$ must be distinct as well, i.e., $d_r(t)\ne d_r(t')$ for some $r\in\{0,1,\ldots,m\}$.
Since $\dim(V_1)\ge2$, there exist $u, v\in L_1$ such that $\bar u$ and $\bar v$ are linearly independent.
By Lemma~\ref{lem:ij} (ii), we have
\[ (z_0\star z_1\star \cdots \star z_m)_t \ne  (z_0\star z_1\star \cdots \star z_m)_{t'} \]
if $z_r:=\bar u$ and $z_s:=\bar v$ for all $s\ne r$, as $c$ generates an infinite multiplicative group in $\RR$.
Thus $t$ and $t'$ are not $\star$-equivalent.
\end{proof}

\subsection{Grassmann graphs}
In this subsection we study the Norton product $\star$ on the eigenspace $V_1$ of the Grassmann graph $J_q(n,k)$.
The case $k=1$ is already covered in the Johnson case as $J_q(n,1) \cong K_{[n]_q} \cong J([n]_q,1)$.
Thus we assume $n\ge2k\ge4$.

Define $\bar v := \frac{[n]_q}{[n]_q-2[k]_q} \check v$ for all $v\in L_1$.
By the formula~\eqref{eq:NortonProd2}, if $u,v\in L_1$ then  
\begin{equation}\label{eq:NortonProd2'}
\bar{u} \star \bar{v} = 
\begin{cases} 
\bar v & \text{if } u=v \\
\displaystyle c(\bar{u}+\bar{v}) +  \sum_{w\in L_1:\,w\le u\vee v} b \bar{w} & \text{if } u\ne v
\end{cases}
\end{equation}
where
\[ c:= -\frac{[k]_q}{[n]_q-2[k]_q} \qand
b:= \frac{[k-1]_q[n]_q}{q[n-2]_q([n]_q-2[k]_q)}. \]

\begin{lemma}\label{lem:Grassman}
Let $t\in\T_m$. 
Let $u$ and $v$ be distinct elements of $L_1$.

\noindent(i)
If $z_0=z_1=\cdots=z_m=\bar v$ then $(z_0\star z_1\star\cdots\star z_m)_t=\bar v$.

\noindent(ii) 
Let $z_r:=\bar u$ for some $r$ and $z_s=\bar v$ for all $s\ne r$.
Let $h:=d_r(t)$.
Then 
\[ (z_0\star z_1\star\cdots\star z_m)_t = \alpha(h) \bar u + \beta(h) \bar v + \gamma(h) \sum_{w\in L_1:\,w\le u\vee v} \bar w \]
where $\alpha(h), \beta(h), \gamma(h)$ are all constants depending only on $n,k,q,h$, with
$\alpha(h) = c^h$ and 
\[ \gamma(h) = \sum_{j=1}^h q^{j-1} \binom{h}{j}b^jc^{h-j}
= \frac{(qb+c)^h-c^h}q. \]
\end{lemma}

\begin{proof}
(i) This follows immediately from the formula $\bar v\star\bar v = \bar v$.

(ii) We use induction on $m$. The result is the same as the formula~\eqref{eq:NortonProd2'} when $m=1$.
For $m\ge2$, let $t$ be a binary tree in $\T_m$ with two subtrees $t_1\in\T_{m'}$ and $t_2\in\T_{m-m'-1}$ rooted at the left and right children of the root of $t$, respectively.
Suppose that the $r$th leaf of $t$ is contained in $t_1$, without loss of generality.
By part (i) of this lemma we have
\[ (z_0\star z_1\star\cdots\star z_m)_t 
= (z_0\star\cdots\star z_{m'})_{t_1} \star (z_{m'+1}\star \cdots\star z_m)_{t_2}
= (z_0\star\cdots\star z_{m'})_{t_1} \star \bar v. \]
Let $h:=d_r(t_1)$. Then $d_r(t)=h+1$.
Applying the inductive hypothesis to $t_1$ we obtain
\begin{align*}
& (z_0\star z_1\star\cdots\star z_m)_t \\
=& \alpha(h) \bar u\star \bar v + \beta(h) \bar v\star\bar v + \gamma(h) \sum_{w\in L_1:\,w\le u\vee v} \bar w\star \bar v \\
=& \alpha(h) c (\bar u+\bar v) + \sum_{\substack{w\in L_1 \\ w\le u\vee v}} \alpha(h) b \bar w + \beta(h) \bar v + \gamma(h) \bar v +
\sum_{\substack{w\in L_1\setminus\{v\} \\ w\le u\vee v}} \gamma(h) c(\bar w+\bar v) + \sum_{\substack{w\in L_1\setminus\{v\} \\ w\le u\vee v}} \sum_{\substack{\tau\in L_1 \\ \tau \le w\vee v}} \gamma(h) b\bar \tau \\
=& \alpha(h) c \bar u + \left(\alpha(h) c + \beta(h) + \gamma(h) + q\gamma(h) c \right) \bar v + 
\sum_{\substack{w\in L_1 \\ w\le u\vee v}} \alpha(h) b \bar w + 
\sum_{\substack{w\in L_1\setminus\{v\} \\ w\le u\vee v}} \gamma(h) c\bar w + 
\sum_{\substack{\tau\in L_1 \\ \tau\le u\vee v}} q\gamma(h) b \bar \tau \\
=& \alpha(h) c \bar u + \left(\alpha(h) c + \beta(h) + \gamma(h) + (q-1)\gamma(h) c \right) \bar v + 
\sum_{\substack{w\in L_1 \\ w\le u\vee v}} (\alpha(h) b+ \gamma(h) c + q\gamma(h) b)\bar w.
\end{align*}
We have $\alpha(h) c = c^h c = c^{h+1} = \alpha(h+1)$ and 
\[\begin{aligned}
\alpha(h) b+ \gamma(h) c + q\gamma(h) b
&= bc^h+ \sum_{j=1}^h q^{j-1} \binom{h}{j}b^jc^{h-j+1} + \sum_{j=1}^h q^{j} \binom{h}{j}b^{j+1}c^{h-j} \\
&= \sum_{j=1}^h q^{j-1} \binom{h}{j}b^jc^{h-j+1} +  \sum_{j=1}^{h+1} q^{j-1} \binom{h}{j-1}b^{j}c^{h+1-j} \\
&= \sum_{j=1}^{h+1} q^{j-1} \binom{h+1}{j}b^jc^{h+1-j} = \gamma(h+1).
\end{aligned} \]
Thus $(z_0\star z_1\star\cdots \star z_m)_t$ satisfies the desired formula.
(It is tedious to determine $\beta(h)$ and we will not need it anyway.)
\end{proof}

\begin{theorem}\label{thm:Grassman}
If $n\ge 2k\ge 4$ then the Norton product $\star$ on the eigenspace $V_1$ of the Grassmann graph $J_q(n,k)$ is totally nonassociative.
\end{theorem}

\begin{proof}
It suffices to show that any two distinct binary trees $t$ and $t'$ in $\T_m$ are not $\star$-equivalent.
By Proposition~\ref{prop:depth}, the depth sequences $d(t)$ and $d(t')$ must be distinct, i.e., $d_r(t)=h$ and $d_r(t')=h'$ are distinct for some $r\in\{0,1,\ldots,m\}$.
Toward a contradiction, suppose that 
\begin{equation}\label{eq:Grassman}
(z_0\star z_1\star \cdots\star z_m)_t = (z_0\star z_1\star\cdots\star z_m)_{t'}.
\end{equation}

Let $u,v$ be distinct elements of $L_1$.
Since $\dim(V_1)=|L_1|-1$, deleting any element from the spanning set $\{\bar w: w\in L_1\}$ gives a basis for $V_1$.
In particular, there exists a subset $L'_1\subseteq L_1$ such that $\{ \bar w:w\in L'_1\}$ is a basis of $V_1$ and $\{w\in L_1: w\le u\vee v\}\subseteq L'_1$.
The set $\{w\in L_1: w\le u\vee v\}$ contains at least three distinct elements $u,v,\tau$, since its cardinality is $1+q\ge3$.

Let $z_r=\bar u$ for some $r$ and $z_s=\bar v$ for all $s\ne r$.
By Lemma~\ref{lem:Grassman}, taking the coefficients of the basis elements $\bar u$ and $\bar \tau$ in the above equation~\eqref{eq:Grassman} gives 
\[ \alpha(h) + \gamma(h) = \alpha(h')+\gamma(h') \qand \gamma(h) = \gamma(h'). \]
Thus $\alpha(h) = \alpha(h')$, i.e., $c^h = c^{h'}$.
Since $h\ne h'$ and $c\in \RR$, we must have $c=\pm1$.
But 
\[ c:=-\frac{[k]_q}{[n]_q-2[k]_q}
= \frac{1-q^k}{2(1-q^k)-(1-q^n)}
= \pm1 \]
implies $q^n-q^k=0$ or $q^n-3q^k+2=0$, which is impossible as $q^n \ge q^{2k} > 3q^k >q^k$ whenever $n\ge 2k\ge4$ and $q\ge 2$.
This contradiction concludes the proof.
\end{proof}

\subsection{Hamming graphs}
In this subsection we study the Norton product $\star$ on the eigenspace $V_1$ of the Hamming graph $\Gamma=H(d,e)$.
When $e=2$, this product is associate since it is acutally zero by Theorem~\ref{thm:Hamming} (see also work of Maldonado and Penazzi~\cite{NortonAlgebra}).
Assume $e\ge3$ below.

\begin{proposition}\label{prop:H(d,3)}
Let $\star$ be the Norton product on the eigenspace $V_1$ of the Hamming graph $H(d,3)$.
Then two binary trees in $\T_m$ are $\star$-equivalent if and only if their depth sequences are (term-wise) congruent modulo $2$.
Consequently, $C_{\star,m}=C_{\ominus,m}$ for all $m\ge0$, which agrees with the sequence A000975 on OEIS~\cite{OEIS} except for $m=0$.
\end{proposition}

\begin{proof}
If $d=1$ then the result follows from Proposition~\ref{prop:J31} since $H(1,3)\cong K_3\cong J(3,1)$.
If $d\ge2$ then we can apply Lemma~\ref{lem:OpTimes} and Corollary~\ref{cor:Hamming} to conclude the proof.
\end{proof}

\begin{proposition}
Let $\star$ be the Norton product on the eigenspace $V_1$ of the Hamming graph $H(d,e)$ with $e\ge4$.
Then two binary trees in $\T_m$ are $\star$-equivalent if and only if they are equal. 
Consequently, $C_{\star,m}=C_{m}$ for all $m\ge0$, i.e., $\star$ is totally nonassociative.
\end{proposition}

\begin{proof}
If $d=1$ then the result follows from Proposition~\ref{prop:J4} since $H(1,e) \cong K_e \cong J(e,1)$.
If $d\ge2$ then we can apply Lemma~\ref{lem:OpTimes} and Corollary~\ref{cor:Hamming} to conclude the proof.
\end{proof}

\subsection{Dual polar graphs}
Finally, we study the Norton product $\star$ on the eigenspace $V_1$ of a dual polar graph $\Gamma=(X,E)$ with diameter $d$.
If $d=1$ then $\Gamma$ is a complete graph, which has already been discussed in the Johnson case. 
Thus we assume $d\ge2$ below.
For every $v\in L_1$, let
\[ \bar v := 
\frac{(q^{d+e-1}+1) \check v}{q^{d+e-1}-1}. \]
By the previous formula~\eqref{eq:DualPolar} for $\star$, if $u,v\in L_1$ then 
\begin{equation}\label{eq:DualPolar'}
\bar{u} \star \bar{v} = 
\begin{cases}
\bar v & \text{if } u=v \\
c(\bar u+\bar v) & \text{if } u\vee v =\hat 1 \\
\displaystyle c(\bar u + \bar v) + b \sum_{w\in\Psi_2} \bar w + \sum_{w\in\Psi_3} b'\bar w 
& \text{otherwise}
\end{cases} 
\end{equation}
where $\Psi_j := \{w\in L_1: u\vee v\vee w\in L_j\}$, $c:= 1/(1-q^{d+e-1})$, 
\[ b:= \frac{(q^{d+e-1}+1)}{(q^{d+e-1}-1)q^{d-1}(1+q^{e-1})} \qand
b':= \frac{(q^{d+e-1}+1)}{(q^{d+e-1}-1)q^{d-1}(1+q^{e-1})(1+q^{d-3+e})}.\]

\begin{example}\label{example2}
A dual polar graph of diameter $d=2$ is a generalized quadrangle of order $(q^e,q)$; in particular, $D_2(q)$ is isomorphic to the complete bipartite graph $K_{1+q,1+q}$~\cite[\S7]{vanBonBrouwer}.
In Example ~\ref{example1} we discussed the complete graphs.
Now we examine the complete $m$-partite graph $K_{n,\ldots,n}$, which becomes the complete graph $K_m$ when $n=1$ or the dual polar graph $D_2(q) \cong K_{1+q,1+q}$ when $m=2$ and $n=1+q$.
For $n\ge2$ the complete $m$-partite graph $K_{n,\ldots,n}$ is a distance regular graph of diameter $d=2$ with adjacency matrix $A=J_{mn} - \mathrm{diag}(J_n,\ldots,J_n)$, where $J_n$ denote the all-one matrix of size $n$-by-$n$. 
The eigenvalues of $A$ are $\theta_0=mn$, $\theta_1=0$, and $\theta_2=-n$.
The eigenspaces are $V_0=\RR\mathbf{1}$,
$V_1 = \{ (v_{11},\ldots,v_{1n},\ldots, v_{m1},\ldots,v_{mn}): v_{i1}+\cdots+v_{in} = 0 \text{ for } i=1,\dots,m \}$, and
\[ V_2 = \{ (v_{11},\ldots,v_{1n},\ldots, v_{m1},\ldots,v_{mn}): v_{i1}=\cdots=v_{in} \text{ for } i=1,\dots,m,\ v_{11}+\cdots+v_{m1}=0 \}.\]
The Norton algebra $V_1$ of the complete $m$-partite graph $K_{n,\ldots,n}$ is isomorphic to the direct product of $m$ copies of the Norton algebra $V_1$ of the complete graph $K_n$ and also isomorphic to the Norton algebra $V_1$ of the Hamming graph $H(m,n)$ by Corollary~\ref{cor:Hamming}.
\end{example}

\begin{theorem}
Let $\star$ be the Norton product on the eigenspace $V_1$ of a dual polar graph of diameter $d\ge2$.
\begin{itemize}
\item
If $\Gamma=D_2(2)$ then two binary trees in $\T_m$ are $\star$-equivalent if and only if their depth sequences are (term-wise) congruent modulo $2$, and consequently, $C_{\star,m}=C_{\ominus,m}$ for all $m\ge0$, which agrees with the sequence A000975 on OEIS~\cite{OEIS} except for $m=0$.
\item
If $\Gamma\ne D_2(2)$ then the operation $\star$ is totally nonassociative.
\end{itemize}
\end{theorem}

\begin{proof}
If $\Gamma=D_2(2)$ then the result follows from Proposition~\ref{prop:H(d,3)} since the Norton algebra $V_1$ of $D_2(2)\cong K_{3,3}$ is isomorphic to that of the Hamming graph $H(2,3)$ as discussed in Example~\ref{example2}.

Suppose $\Gamma\ne D_2(2)$ below.
There exist distinct elements $u,v\in L_1$ such that $u\vee v=\hat1$, i.e., the span of $u\cap v$ is not isotropic, as otherwise the span of all isotropic one-dimensional subspaces would be the unique maximal isotropic subspace, giving a contradiction to the hypothesis $d\ge2$.

Since $u\vee v=\hat1$, the formula~\eqref{eq:DualPolar'} gives $\bar u \star \bar v = c(\bar u + \bar v)$ where $c:= 1/(1-q^{d+e-1})$.
Thus Lemma~\ref{lem:ij} still holds and we can argue in the same way as the proof of Proposition~\ref{prop:J4} to show that $\star$ is totally nonassociative, provided that (i) $\bar u$ and $\bar v$ are linearly independent, and (ii) $c\ne \pm1$. 

To show the assumption (i), suppose that $\bar u + \lambda\bar v=0$ for some constant $\lambda$, for the sake of contradiction.
This implies that $\pi_1(\imath_u+\lambda\imath_v)=0$, i.e., $\imath_u+\lambda \imath_v\in \Lambda_0 = \RR \mathbf{1}$.
Thus any $x\in X$, we have either $u\le x$ or $v\le x$, but not both since $u\vee v =\hat 1\not\le x$.
In other words, $\imath_u+\lambda\imath_v= \imath_u+\imath_v=\mathbf{1}$, which implies that $\bar u +\bar v=0$.
Then we have a contradiction between $\bar u \star \bar v = c(\bar u+ \bar v) = 0$ and $\bar u \star \bar v = - \bar v \star \bar v = - \bar v\ne0$.

Now suppose that the assumption (ii) is false, i.e., $c:= 1/(1-q^{d+e-1})=\pm1$.
This implies that $q^{d+e-1}=0$ or $q^{d+e-1}=2$.
The former is absurd as $q\ge2$.
The latter holds if and only if $q=d=2$ and $e=0$, but in this case the dual polar graph is exactly $D_2(2)$.
Thus the proof is complete.
\end{proof}


\section{Remarks and questions}\label{sec:remark}

\subsection{Explicit formula for the Johnson graphs}

For the Norton product $\star$ on the eigenspace $V_1$ of the Johnson graph $J(n,k)$, we can use the formula~\eqref{eq:NortonProd1'} for $\star$ to simplify the expression $(z_0\star z_1\star\cdots\star z_m)_t$, where $t$ is a binary tree with $m+1$ leaves and $z_0,z_1,\ldots,z_m$ are indeterminates taking values in $V_1$, or equivalently, in the spanning set $\{\bar v: v\in L_1\}$ of $V_1$. 
It would be nice to have an explicit rule for the result, especially in the case $(n,k)=(3,1)$ when the formula~\eqref{eq:NortonProd3} for $\star$ is relatively simple.

\subsection{Generalized Johnson graphs}

The \emph{generalized Johnson graph} $J(n,k,r)$ has vertex set $X$ consisting of all $k$-subsets of $[n]$ and has edge set
\[ E=\{ xy: x, y\in X,\ |x \cap y|=k-r\}. \]
It includes the Johnson graph $J(n,k)=J(n,k,1)$ as a special case.
The eigenvalues of the adjacency matrix of $J(n,k,r)$ are given by the so-called \emph{Eberlein polynomials};
this can be derived in terms of association schemes ~\cite{BannaiIto,Delsarte} or by representation theoretic means~\cite{Spec1,Spec2}.
The distance in $J(n,k,r)$ was determined by Agong, Amarra, Caughman, Herman and Terada~\cite{GenJohnson}.
In general, the graph $J(n,k,r)$ is not distance regular.
Thus we cannot extend our results from $J(n,k)$ to $J(n,k,r)$.

\subsection{Bilinear Forms Graphs}

The \emph{bilinear forms graph} $H_q(d,e)$ has vertex set $X$ consisting of all $d\times e$ matrices over a finite field $\FF_q$ and has edge set $E$ consisting of all unordered pairs of vertices whose difference has rank one.
Two vertices have distance $i$ in $H_q(d,e)$ if and only if their difference has rank $i$.
The graph $H_q(d,e)$ is a distance regular graph of diameter $d$ (assuming $d\le e$) and is a $q$-analogue of the Hamming graph $H(d,e)$.
It would be nice to have an explicit formula for the Norton product on the eigenspace $V_1$ of the graph $H_q(d,e)$.

\subsection{Other distance regular graphs}
There are many other interesting distance regular graphs in the literature; see, for example, Brouwer, Cohen and  Neumaier~\cite{DistReg1} and van Dam--Koolen--Tanaka~\cite{DistReg2}.
The Norton algebras of these graphs are worth further investigation.

\section*{Acknowledgments}

This work originated from a visit to the University of Wisconsin;
the author is grateful to Paul Terwilliger for his invitation and hospitality and for inspiring conversations with him on the Norton algebras during the visit.
The author also thanks Fernando Levstein for helpful discussions on the Norton algebras of the Hamming graphs, and thanks the anonymous referee for valuable suggestions on the paper.
Finally, this work benefits from computations that the author performed in SageMath.

\end{document}